\documentclass[11pt, reqno]{amsart}

\title[$\mathcal{L}^{p}$- and pointwise convergence of families of functions]{Relations between $\mathcal{L}^{p}$- and pointwise convergence of families of functions indexed by the unit interval.}

\newcommand{\Kl}{\left(}
\newcommand{\Kr}{\right)}

\newcommand{\Ra}{\rightarrow}

\newcommand{\hh}{\hspace{4pt}}

\newcommand{\NN}{\mathbb{N}}
\newcommand{\RR}{\mathbb{R}}

\newcommand{\gG}{\gamma}

\newcommand{\dD}{\delta}

\newcommand{\eE}{\varepsilon}

\newcommand{\oO}{\omega}
\newcommand{\OO}{\Omega}
\newcommand{\fF}{\varphi}
\newcommand{\lL}{\lambda}
\newcommand{\sS}{\sigma}

\newcommand{\Ii}{\mathcal{I}}

\newcommand{\Cc}{\mathcal{C}}

\newcommand{\di}{\mathrm{d}}

\newtheorem{theorem}{Theorem}[section]
\newtheorem{lemma}[theorem]{Lemma}
\newtheorem{proposition}[theorem]{Proposition}
\newtheorem{corollary}[theorem]{Corollary}
\newtheorem{remark}[theorem]{Remark}

\begin{document}

\begin{abstract}
We construct a variety of mappings from the unit interval $\mathcal{I}$ into $\mathcal{L}^p([0,1]),1\leq p<\infty,$ to generalize classical examples of $\mathcal{L}^p$-converging sequences of functions with simultaneous pointwise divergence. By establishing relations between the regularity of the functions in the image of the mappings and the topology of $\mathcal{I}$, we obtain examples which are $\mathcal{L}^p$-continuous but exhibit discontinuity in a pointwise sense to different degrees. We conclude by proving a Lusin-type theorem, namely that if almost every function in the image is continuous, then we can remove a set of arbitrarily small measure from the index set $\mathcal{I}$ and establish pointwise continuity in the remainder.
\end{abstract}
\maketitle
\section{Introduction}

\subsection{Motivation and Overview}
Examples of sequences of real functions on a compact domain which have a limit in $\mathcal{L}^p$, but do not converge pointwise are well known. Their construction is based on the fact that any interval can be covered infinitely often by a sequence of subintervals of vanishing lengths. Take, for instance, the sequence $(f_i)_{i\in \mathcal{I}}$ of characteristic functions $f_i=\chi_{I_i},i\in \mathbb{N}$, of the intervals $I_i=[\frac{i}{2^k}-1,\frac{i+1}{2^k}-1],$ where $k$ is the unique integer with $2^k\leq i<2^{k+1}$. After applying a suitable molifier to smoothen each member of the sequence, we can see that this pointwise divergence is not affected by the smoothness of the functions. In such examples, only the order of the index set is relevant. We can observe, however, that a simple topology is induced in a natural way by the convergence of the sequence. It is not obvious whether examples of this type can be extended to index sets of a more complex topological structure. We wish to address the case of a 
continuous curve $f$ which maps $\mathcal{I}=[0,1]$ into $\mathcal{L}^p([0,1])$ and generalize examples like the above. In our setting, the index set $\mathcal{I}$ has a non-trivial topological structure of its own, which turns out to interact with the regularity properties of the family $\{f_t,t\in\mathcal{I}\}.$\\

Curves such as $\{f_t,t\in\mathcal{I}\}$ often appear in semigroup theory as solutions of PDEs. However, the smoothing properties of the operators in these settings usually result in a high regularity for the solutions for every $t>0$, and therefore pointwise convergence comes naturally. Even for the more anomalous case of $t=0$, pointwise convergence can often be deduced by using tools from harmonic analysis or potential theory. In this paper, no underlying process is assumed. We investigate the pointwise behaviour of the curves in a purely real analytic way.\\

By making different assumptions regarding the properties of the functions $f_t$, we construct two example curves in $\mathcal{L}^p$ which lack pointwise convergence almost everywhere. The first example is constructed in Section~2, where we assume that $\{f_t,t\in\Ii\}\subset C(\OO)$ and that almost all $f_t$ are smooth. In Section~3 we then show that the criteria on the regularity $f_t$ are optimal. We demonstrate that the structure of $\Ii$ renders ``everywhere pointwise divergence'' impossible, and that higher regularity always implies better pointwise convergence properties.\\

In Section~4 we remove the continuity requirement and construct a curve of highly irregular functions. For this curve, we not only have everywhere pointwise divergence, but also, for every subset $T$ of $\mathcal{I}$ with positive measure, the restriction $f_{|T}$ exhibits pointwise divergence almost everywhere.\\

Finally, Section~5 is devoted to proving that the discontinuity of $f_{t}$ is necessary to obtain a curve that exhibits such a highly pointwise divergence. In particular, the example in Section~4 motivates a special case of our main result Theorem~5.2., which can be interpreted as a refined version of Lusin's Theorem in two variables.
 
\subsection{Notation}
Throughout, $\mathcal{I}$ is the unit interval $[0,1]$, equipped with the standard norm $|\cdot|$ and the corresponding Borel-$\sS$-field. Lebesgue measure on $\mathcal{I}$ is denoted by $\mu$. We study functions $f:\Ii\times\OO\longrightarrow \RR$ of two real variables, we will usually, for $t\in \Ii$, write $f_t$ for the function $f(t,\cdot)$ in one real variable to stress the difference between ``time'' and ``space'', but revert to write $f$ as a function of two variables when it is notationally more convenient. The spacial domain $\OO\subset\RR$ of the functions $f_t,t\in\mathcal{I}$, can be chosen to be any interval of $\RR$ equipped with its Borel $\sS$-field and Lebesgue measure. We take $\OO=[0,1]$ for convenience in the construction of the examples. We denote Lebesgue measure by $\lambda$ to avoid confusion with the "time`` interval $\mathcal{I}$, whenever we refer to space, i.e. when measuring sets in the domain and range of the real functions $f_t,t\in [0,1]$.\\

$\mathcal{L}^p(\OO,\RR,\lL)=\mathcal{L}^p$, for $1\leq p<\infty$, denotes the space of real-valued $p$-integrable functions on $\OO$, equipped with the topology induced by the seminorm $\|\cdot\|_p.$ Furthermore, we write $\mathcal{W}^{1,p}$ for the space of all absolutely continuous functions with derivatives belonging to $\mathcal{L}^{p}$ and use the standard notation $C(\OO)$ and $C^{\infty}(\OO)$ for the space of continuous real valued functions on $\OO$ and the space of real valued smooth functions on $\OO\setminus\partial\OO$.  

\begin{remark}
Note that we do not identify almost everywhere indentical members of $\mathcal{L}^p$, since all our constructions are pointwise. To prove lack of convergence at a point $t$, we choose a sequence ${t_{n}}$  converging to $t$ and assure, that $f_{t_n}$ diverges pointwise on a set of positive measure. Therefore the established irregularity can not be avoided by choosing different ''versions`` of $f_t$ and trivial counterexamples like the continuous transport of a set of measure zero are excluded.
\end{remark}

\section{Construction of the first example}
We begin by showing that there is a $\mathcal{L}^p$-continuous curve of continuous functions, along which pointwise convergence can be established almost nowhere. Moreover, this irregularity is achieved while keeping almost all functions along the curve smooth.

\begin{theorem}\label{Ex1}
	Let $1\leq p<\infty$ and $K\subset[0,1]$ be a meager $F_\sigma$. There is a continuous mapping $f$ of $[0,1]$ into $\mathcal{L}^p(\OO)$, satisfying\vspace{6pt}\\
	(i) $f_t$ is absolutely continuous for all $t\in[0,1]$,\vspace{6pt}\\
	(ii) $f_t\in C^\infty(\OO)$ for every $t\in K$,\vspace{6pt}\\
	but also\vspace{6pt}\\
	(A) for every $t\in K$ there exists a sequence $(t_n)_{n\in\NN}$ with $\lim_{n\Ra\infty}t_n=t$ such that $$\lambda\big(\{x\in\OO:\;(f_{t_{n}}(x))_{n\in\NN}\textrm{ is Cauchy}\}\big)=0. $$ In particular, if $\mu (K)=1$, then the conditions in (ii) and (A) hold for $\mu$-a.e. $t\in[0,1].$
\end{theorem}
\begin{proof}
 	Without loss of generality, we will assume that $\{0,1\}\subset K$. We can represent $K=\bigcup_{i=1}^\infty K_{i}$, where $\{0,1\}\subset K_1\subset K_2\subset\dots$ are closed nowhere dense subsets of $[0,1]$. For each $i$, the complement $K_i^C$ can be represented as a countable union $\bigcup_{j=1}^\infty (r_{i,j},s_{i,j})$ of disjoint open intervals, whose lengths we denote by $l_{i,j}=\mu((r_{i,j},s_{i,j})).$ In this setting define
				$$f^{(i)}(t,x)=\fF_i(t)\gG_{i}(t,x),$$
	where		$$\fF_i(t)=	\begin{cases}
						\frac{2j}{l_{i,j}}(t-r_{i,j})	&\mbox{if }t\in(r_{i,j},r_{i,j}+\frac{l_{i,j}}{2j}),\\
						1	&\mbox{if }t\in[r_{i,j}+\frac{l_{i,j}}{2j},s_{i,j}-\frac{l_{i,j}}{2j}],\\
						\frac{2j}{l_{i,j}}(t-s_{i,j})	&\mbox{if }t\in(s_{i,j}-\frac{l_{i,j}}{2j},s_{i,j}),\\
						0	&\mbox{otherwise}	
						\end{cases}$$
	and			$$\gG_{i}(t,x)=	\begin{cases}
							\frac{1}{4^i}\exp\Kl \frac{-\pi\Kl x-\frac{t-r_{i,j}}{l_{i,j}}\Kr^{2}}{l_{i,j}^{2p}} \Kr	&\mbox{if }t\in(r_{i,j},s_{i,j})\mbox{ for some }j\in\NN\\
								&\text{ and}\hh x\in[0,1],\\
							0	&\mbox{otherwise.}			
						\end{cases}$$
	A straightforward calculation shows that, for all $i\in\NN$, $f^{(i)}(t,x)\leq 4^{-i}$ for all $(t,x)\in[0,1]\times\Omega$, and thus $\|f^{(i)}(t,\cdot)\|_p\leq4^{-i}$ for all $t\in[0,1]$. For fixed $i\in\NN$ we next show $\mathcal{L}^p$-continuity of $f^{(i)}$ in the first variable. We will prove it for $t$ being approximated from the right and the rest can be proved analogously. Application of the triangle inequality and convexity of $(\cdot)^p$ yield that, for all $t, u\in[0,1]$,
	\begin{equation}\label{eq:lp_cont}
	\begin{split}
	& \|f^{(i)}(t,\cdot)-f^{(i)}(u,\cdot)\|_p^{p} \\
	& \leq2^{p-1}\big(|\varphi_i(t)-\varphi_i(u)|^p\|\gamma_i(t,\cdot)\|_p^p+|\varphi_i(u)|^p\|\gamma_i(t,\cdot)-\gamma_i(u,\cdot)\|_p^p\big).
	\end{split}
	\end{equation} We now distinguish three cases. Firstly, if $t\in(r_{i,j},s_{i,j})$ for some $j\in\NN$, then the right hand side of \eqref{eq:lp_cont} vanishes as $u$ converges to $t$, since $\varphi_i$ is continuous and $\gamma_i(t,\cdot)$ is $\mathcal{L}^p$-continuous in $t$. Secondly, if $t=r_{i,j}$ for some $j\in\NN$, then $\lim_{u\to t}\varphi(u)=\varphi(t)=0$. Thus the right hand side of \eqref{eq:lp_cont} can be made arbitrarily small by chosing $u$ close to $t$, since it is bounded by a positive multiple of $|\varphi(u)|^p$, due to the uniform boundedness of $\gamma_i$. Finally, if $t$ is not contained in any of the intervals $\{[r_{i,j},s_{i,j})\}_{j\in\NN}$, then either $f^{(i)}(u,\cdot)\equiv 0$ for all $u$ in a set of the form $[t,t+\epsilon)$ or $t$ is an accumulation point, from the right, of a subsequence $((r_{i,j_k},s_{i,j_k}))_{k\in\NN}$ of nonempty intervals with $\lim_{k\to\infty}l_{i,j_k}=0$. In the latter case, we can assume w.l.o.g. that $(s_{i,j_k})_{k\in\NN}$ is montonically 
decreasing and that $u$ in \eqref{eq:lp_cont} is an element of $(t,s_{i,j_k})$. Both $\gamma_i$ and $\varphi_i$ are uniformly bounded and $\gamma_i(t,\cdot)\equiv 0$. The sum in \eqref{eq:lp_cont} is therefore bounded by a constant multiple of $\mu((t,s_{i,j_k}))$. As $u$ approximates $t$ from the right, $k$ can be chosen larger and the bound can be made arbitrarily small, since $\lim_{k\to\infty}s_{i,j_k}=t$. This shows that $f^{(i)}$ is $\mathcal{L}^p$-continuous from the right in $t$.  Combining all three cases, $f^{(i)}(t,\cdot)$ is thus $\mathcal{L}^{p}$-continuous for all $t$ in the compact interval $[0,1]$ and is therefore also uniformly continuous on $K_{i}^{C}$.
 We set $$f_t(x)=\sum_{i=1}^\infty f^{(i)}(t,x),$$ which defines a function $f:[0,1]\longrightarrow \mathcal{L}^{p}(\OO)$ for which the properties stated in the theorem can be verified. To show that $f$ is a continuous mapping of the unit interval into $(\mathcal{L}^p(\OO),\|\cdot\|_p)$ we fix $\eE>0$, choose $l\in\NN$ such that $2^{-l}<\frac{\eE}{2}$ and estimate for $t,u\in\mathcal{I}$, using the triangle inequality and Fatou's Lemma
	\begin{align*}
	 	&\Kl \int_\OO |f_u(x)-f_t(x)|^{p}\di x \Kr^{\frac{1}{p}}\\
		&=\Kl\int_\OO\Big|\sum_{i=1}^\infty\Kl f^{(i)}(u,x)-f^{(i)}(t,x)\Kr\Big|^{p}\di x\Kr^{\frac{1}{p}}\\
		&\leq\sum_{i=1}^\infty\Kl\int_\OO |f^{(i)}(u,x)-f^{(i)}(t,x)|^{p}\di x \Kr^{\frac{1}{p}}\\
		&=\sum_{i=1}^l\Kl \int_\OO |f^{(i)}(u,x)-f^{(i)}(t,x)|^{p}\di x\Kr^{\frac{1}{p}}\\
		&\hh\hh\hh+\sum_{i=l}^\infty\underbrace{\Kl \int_\OO |f^{(i)}(u,x)-f^{(i)}(t,x)|^{p}\di x\Kr^{\frac{1}{p}}}_{\leq\frac{2}{4^i}}.
	\end{align*}
	The finite left hand side summand can be made smaller than $\frac{\eE}{2}$ by choosing $u$ close to $t$ and the right hand side summand is bounded by $2^{-l}$ and therefore by $\frac{\eE}{4}$. Since $t$ and $\eE$ can be chosen arbitrarily, we conclude that $\mathcal{L}^p$-continuity holds for every $t\in[0,1]$.\\ 
	
	Recalling $K=\bigcup_{i=1}^\infty K_{i}$, for any $t\in K$ there exists an index $i$ for which $t\in K_i$. Hence $f(t,\cdot)$ is a finite sum of $\Cc^\infty(\OO)$-functions and therefore smooth, so (ii) holds. For (i), absolute continuity only needs to be verified for $t\in K^C$. Clearly all $f^{(i)}(t,\cdot)$ are absolutely continuous and so are the finite sums $\sum_{i=1}^kf^{(i)}(t,\cdot)$. For existence of the derivative $\frac{\di}{\di x}f(t,x),$ we note that
	\begin{align*}\label{diffcon}
	  &\frac{\di}{\di x}\frac{1}{4^i}\exp\Kl \frac{-\pi\Kl x-\frac{t-r_{i,j}}{l_{i,j}}\Kr^{2}}{l_{i,j}^{2p}} \Kr\\
	  &=-\frac{1}{4^i}\frac{2\pi}{l_{i,j}^{2p}}\Kl x- \frac{t-r_{i,j}}{l_{i,j}}\Kr \frac{1}{4^i}\exp\Kl \frac{-\pi\Kl x-\frac{t-r_{i,j}}{l_{i,j}}\Kr^{2}}{l_{i,j}^{2p}} \Kr,
	 \end{align*}
	which is still summable in $i$. Hence the first derivative of $f_t$ exists and it is continuous for all $t\in K^C$ a.e. on $\OO$, which implies absolute continuity.\\
	
	The next step is to show that (A) holds. We do so by constructing for every $t\in K$ a sequence $(t_{n})_{n\in\NN}$ such that $f(t_n,\cdot)$ has the desired property. We observe first, that if we fix $i,j\in\NN$, $x\in[\frac{1}{j},1-\frac{1}{j}]$ and set $\tau_x=r_{i,j}+xl_{i,j}$, then continuity of $f(\tau_x,\cdot)$ implies that the sets $I_x=\{y:f(\tau_x,y)>\frac{2}{3}4^{-1}\}$ are open. The defintions of $f^{(i)}$ and $f$ imply that $f(\tau_x,x)\geq 4^{-i}$, thus $x\in I_x$ and $\bigcup_x I_x$ is an open cover of the interval $[\frac{1}{j},1-\frac{1}{j}]$. By compactness, we can find a finite subcover, i.e. there is an integer $k$ and a $k$-tuple $\tau=\big(\tau^{1},\tau^{2},...,\tau^{k}\big)$, where $\tau^l\in(r_{i,j},s_{i,j})$, for $1\leq l\leq k$, with the property that for every $x\in [\frac{1}{j},1-\frac{1}{j}]$ there exists an index $l(x)\in\{1,\dots,k\}$ such that $f^{(i)}(\tau^{l(x)},x)>\frac{2}{3}4^{-i}$.\\
	
	Now let $t\in K$ be fixed and let $i=i(t)=\min\{j\in\NN:\;t\in K_{j}\}$. Since $K_{i}$ is nowhere dense, there exists a subsequence of intervals $\big((r_{i,j_{n}},s_{i,j_{n}})\big)_{n\in\NN}$ indexed by $(j_n)=(j_n(t))$ with endpoints $r_{i,j_{n}},s_{i,j_{n}}$ converging to $t$. Thus for each one of the intervals $(r_{i,j_{n}},s_{i,j_{n}})$ we can apply the above argument and find $k=k(n)\in\NN$ and a $k$-tuple $\tau(n)=\big(\tau^{1}(n),\tau^{2}(n),...,\tau^{k}(n)\big)$, such that $\tau^l(n)\in(r_{i,j_{n}},s_{i,j_{n}}),1\leq l\leq k,$ and for every $x\in [\frac{1}{j_{n}(t)},1-\frac{1}{j_{n}(t)}]$ there is $l=l(x,n)\in\{1,\dots,k\}$ satisfying $f^{(i)}(\tau^{l}(n),x)>\frac{2}{3}4^{-i}$. Note also that$f^{(i)}(r_{i,j_n},\cdot)=f^{(i)}(s_{i,j_n},\cdot)\equiv 0$.\\

	Finally, we consider now the sequence $(t_{m})_{m\in\NN}$ obtained by concatenating the $k(n)+2$-tuples $(r_{i,j_{n}}, \tau^{1}(n),\tau^{2}(n),\dots,\tau^{k}(n),s_{i,j_{n}})$ in increasing order in $n$. Fix $x_{0}\in[0,1], n_{0}\in\NN $ and $\eE=\frac{1}{6}4^{-i}$. Since $t\notin K_{h}$ for any  $h<i$ we know that the functions $f^{(h)}(\tau,\cdot)$ are uniformly continuous around $t$ for every $h<i$, i.e there is a $\delta>0$ such that $\sum_{h=1}^{i-1}\big(f^{(h)}(t,\cdot)-f^{(h)}(\tau,\cdot)\big)<\frac{1}{4^{i}}$ for every $\tau$ with $|t-\tau|\leq \delta$. Since $t_{n}$ converges to $t$ there is $n_{1}\in\NN$ such that for every $n\in\NN$ with $n>n_{1}$ we have $|t_{n}-t|<\delta$ and by construction of $t_{n}$ there are $n,m >\max\{n_{0},n_{1}\}$ such that $f^{(i)}(t_{m},x_{0})-f^{(i)}(t_{n},x_{0})>\frac{2}{3}4^{-i}$. For these $n,m$ we have
	\begin{align*}
	  f&(t_{m},x)-f(t_{n},x)=\sum_{h=1}^{\infty}(f^{(h)}(t_{m},x)-f^{(h)}(t_{n},x))\\
	  =&\sum_{h=1}^{i-1}(f^{(h)}(t_{m},x)-f^{(h)}(t_{n},x))+f^{(i)}(t_{m},x)-f^{(i)}(t_{n},x)\\
	  &+\sum_{h=i}^{\infty}(f^{(h)}(t_{m},x)-f^{(h)}(t_{n},x))\\
	  \geq&\frac{2}{3\cdot 4^{i}}- \frac{1}{4\cdot 4^{i}}- \sum_{h=i}^{\infty}\frac{1}{4^{h}}=\frac{2}{3\cdot 4^{i}} -\frac{1}{4\cdot 4^{i}}-\frac{1}{4\cdot 4^{i}}=\eE,
	\end{align*}
	so $\big(f(t_{n},x)\big)_{n\in\NN}$ is not Cauchy.
\end{proof}

\section{Optimality of the conditions in Theorem~\ref{Ex1}}
In this section we show that Theorem~\ref{Ex1} is sharp in two senses. Firstly, the following argument shows that $K$ in Theorem~\ref{Ex1} cannot be non-meager, thus the example is best possible in the sense of Baire category. In particular, we cannot obtain divergence on the whole of $\Ii$.

\begin{proposition}
	Let $f$ be a continuous map of $[0,1]$ into $\mathcal{L}^p(\OO)$.	If $f_t$ is continuous for every $t$, then there is a comeagre subset $T\subset[0,1]$ such that for any $t\in I$ and any sequence $(t_n)_{n\in \NN}$ with limit $t$
	$$\lim_{n\to\infty}f_{t_n}(x)=f_t(x)\;\;\textrm{ for all }x\in\OO.$$
\end{proposition}

\begin{proof}
Define for $0<q<p$ the sets
\begin{equation*}
T_{pq}=\big\{t\in[0,1]: \hh \exists \hh x(t)\in\OO \hh\text{with} \hh f_t(x(t))<q<p<\limsup_{s\rightarrow t} f_s(x(t))\big\}.
\end{equation*}
We first want to prove by contradiction that $T_{pq}$ are nowhere dense sets. Let us assume that there are $q<p$ such that $T_{pq}$ is dense in an open ball $B(t_{0},r_{0})$, with $t_0\in T_{pq}$. We then have that no open subset $S$ of the ball $B(t_{0},r_{0})$ is disjoint from $T_{pq}$.\\

We start by demonstrating that for any such $S$ and any choice of $\delta>0$ and sufficiently small $\rho>0$, there exist $t\in S$ and $r<\rho$ such that for every $s\in B(t,r)$ we have $\omega(s,\delta)>q-p,$ where $\omega(s,\delta)=\sup\{|f_s(x)-f_s(y)|:|x-y|<\delta\}.$\\

By assumption, there is ${t_1}\in S\cap T_{pq}$, hence there is a point $x({t_1})\in\OO$ with $f_{{t_1}}(x({t_1}))<q<p<\limsup_{s\rightarrow {t_1}}f_s(x({t_1}))$. Since $f_{t_1}(x({t_1}))<q$, there exists $0<\delta_1<\delta$ such that $f_{t_1}(y)<q$, for all $y\in B(x({t_1}),\delta_1)$, by continuity of $f_{t_1}$. Moreover, $f$ is $\mathcal{L}^p$-continuous, hence there is $r_1>0$ such that for every $s\in B(t_1,r_1)$, there exists $x^u(s)\in B(x(t_1),\delta_1)$ for which $f_s(x^u(s))<q$. We choose now a second point $t_2\in B(t_1,r_1)$ with $f_{t_2}(x(t_1))>p$ and by continuity of $f_{t_2}$ we can fix $\delta_2>0$ such that $f_{t_2}(y)>p$ for all $y\in B(x(t_1),\delta_2)$. Using $\mathcal{L}^p$-continuity again, we can find $r_2>0$ such that for every $s\in B(t_2,r_2)$ there exists $x^l(s)\in B(x(t_2),\delta_2)$ with $f_s(x^l(s))>p.$ The above assertion now holds for the choices $t=t_2$, $r=\min\{r_1,r_2,\sup\{|t_2-s|, s\in \partial B(t_1,r_1)\}\}$ and $\delta=\delta_1.$\\

Applying the above construction to vanishing sequences $(\rho_n)_{n\in\NN},(\delta_n)_{n \in \NN}$, we can find points $t_n$ and radii $r_n<\rho_n$ with $B(t_{n+1},r_{n+1})\subset B(t_{n},r_{n})$ and $\omega(s,\delta_{n})>q-p$ for every $s\in B(t_{n},r_{n})$. Since $\lim_{n\to\infty} r_{n}=0$, we have that  $\lim_{n\to\infty}t_{n}=t_{\infty}$ for some $t_{\infty}\in [0,1]$. Moreover, we have that $\omega(t_{\infty},\delta_{n})>p-q$ for every $n\in \NN$, which contradicts the assumption that $f_t$ is continuous for every $t\in[0,1]$, hence our initial assumption that $T_{pq}$ is not nowhere dense cannot be true.\\

We can apply the same argument to the sets \begin{equation*}
S_{pq}=\big\{t: \hh \exists \hh x(t)\in\OO \hh \text{such that} \hh f_t(x(t))>q>p>\liminf_{s\rightarrow t} f_s(x(t))\big\},
\end{equation*}
and the comeager set $T$ mentioned in the theorem is the complement of $$\bigcup_{p,q\in\mathbb{Q}} (T_{pq}\cup S_{pq}).$$ 

\end{proof}

Secondly, we can prove that we cannot make the regularity requirement (i) in Theorem~\ref{Ex1} stronger. 
\begin{proposition}\label{SmoothSharp}
Let $f$	be a continuous mapping of $[0,1]$ into $\mathcal{L}^p(\Omega)\cap \mathcal{W}^{1,q}(\Omega),$ where $1\leq p<\infty$ and $q>1$. Then there is an open dense set $T\subset[0,1]$ such that for all $t\in T$ and any sequence $(t_n)_{n\in \NN}$ with limit $t$, $$\lim_{n\to\infty}f_{t_n}(x)=f_t(x)\;\textrm{ for all }x\in\Omega.$$
\end{proposition}
For the proof of Proposition~\ref{SmoothSharp}, we need to establish an auxiliary lemma about the relation between $\mathcal{L}^p$-continuity and pointwise continuity.
\begin{lemma}\label{lem1}
	Let $f$ be $\mathcal{L}^p$-continuous and $S\subset[0,1]$ an open interval. If $f_t\in \mathcal{W}^{1,q}(\Omega)$ for some $q>1$ and $\{f_t;t\in S\}$ is bounded in $\mathcal{W}^{1,q}(\Omega)$, then $f$ is pointwise continuous for every $t\in S,$ i.e. $\lim_{n\to\infty}f_{t_n}(x)=f_t(x)$ for every sequence $(t_n)_{n\in\NN}$ converging to $t$ and every $x\in\OO$.
\end{lemma}

\begin{proof}
	Fix $\varepsilon>0$. Since $f_t\in \mathcal{W}^{1,q}(\Omega)$, invoking the Sobolev Imbedding Theorem (see, e.g., \cite[Part II of Theorem 4.12 with $m=n=1,j=0,n=1,p=q$ and $\lambda=1-1/q$]{Ad03}), we can assume that $f_t$ is H\"{o}lder-continuous with exponent $q'=1-\frac{1}{q}$ and constant $C_t>0$ independent of $t$, i.e. we have for all $t \in S$,
	\begin{equation} \label{HoldCont1}
		|f_t(x)-f_t(y)|\leq \, C_t|x-y|^{q'},\hspace{5pt} \;\textrm{ for all }x,y \in\OO.
	\end{equation}
	The proof of this part of the Sobolev Imbedding Theorem (see, e.g., \cite[p. 100, proof of Lemma 4.28]{Ad03}) demonstrates that the H\"older constant $C_t$ is bounded by a constant multiple of $\|f_t\|_{1,q}$, using the boundedness of $\{f_t;t\in S\}$ we can therefore assume that \eqref{HoldCont1} holds uniformly on $S$ with $C_t\equiv C$. Now, fixing $x\in\OO$ and any $s,t\in S$ and then applying the triangle inequality and \eqref{HoldCont1}, we obtain, for all $y \in \OO$,
	\begin{align*}
		|f_t(x)-f_s(x)|\leq&|f_t(x)-f_t(y)|+|f_t(y)-f_s(y)|+|f_s(x)-f_s(y)|\\
		\leq&2C|x-y|^{q'}+|f_t(y)-f_s(y)|.
	\end{align*}
	Integrating both sides in $y$ on the interval $B(x,\frac{\eta}{2})=(x-\frac{\eta}{2},x-\frac{\eta}{2})$, where $0<\eta<\min\{\frac{\varepsilon}{2},2\sqrt[q']{\frac{\varepsilon}{4C}}\}$, yields
	\begin{align*}
	\Kl\int_{B(x,\frac{\eta}{2})}|f_t(x)-f_s(x)|^{p}\di y\Kr^{\frac{1}{p}}\leq& \Kl\int_{B(x,\frac{\eta}{2})}(2C|x-y|^{q'})^{p}\di y\Kr^{\frac{1}{p}}\\ &+ \Kl\int_{B(x,\frac{\eta}{2})}|f_t(y)-f_s(y)|^{p}\di y\Kr^{\frac{1}{p}}
	\end{align*} and thus
	\begin{equation*}
	\eta^{\frac{1}{p}} |f_t(x)-f_s(x)| \leq \eta^{\frac{1}{p}}\varepsilon + \|f_t-f_s\|_{p}.
	\end{equation*}
	This implies $$
	|f_t(x)-f_s(x)|\leq \frac{\varepsilon}{2} + \frac{\|f_t-f_s\|_{p}}{\eta^{\frac{1}{p}}}$$
	and using $\mathcal{L}^{p}$ continuity of $f$ we derive that
	\begin{equation*}
	|f_t(x)-f_s(x)|\leq \varepsilon,
	\end{equation*}
	for all $s$ sufficiently close to $t$.
\end{proof}

We are now ready to prove Proposition~\ref{SmoothSharp}.
\begin{proof}[Proof of Proposition~\ref{SmoothSharp}]
Let $f_t\in \mathcal{W}^{1,q}(\OO),q>1$ for every $t\in[0,1]$. Since $f_t$ is absolutely integrable for all $t$, we can expand each $f_t$ into a Fourier series
$$ f^{(m)}(t,x)=\sum_{n=-m}^{n=m}c_{n}(t)e^{in\pi(x-\frac{1}{2})}$$ for which we have $ \lim_{m\to\infty}f^{(m)}(t,\cdot)=f_t$, w.r.t. $\|\cdot\|_{1,q}$, see e.g. \cite[p. 78]{Ba64}. Using $\mathcal{L}^{p}$-continuity of $f$ we obtain that the coefficients $c_{n}(t)=\int_\OO f_t\exp\big(-in\pi(\cdot-\frac{1}{2})\big)\di \lambda$ are continuous in $t$ and furthermore $g^{(m)}(t)=\|f^{(m)}(t,\cdot)\|_{1,q}$ is a continuous function. Hence $g(t)=\|f_t\|_{1,q}$ can be represented as a limit of continuous functions and therefore the set of points of continuity of $g$ is comeager $G_\delta$, see e.g. \cite[Theorem 7.3]{Ox71}. This implies that $g(t)$ is locally bounded on an open dense set. Thus the assumptions of Lemma~\ref{lem1} are satisfied and its application yields the statement of the theorem. 
\end{proof}

\section{Construction of the second example}
If the requirement~(i) in Theorem~\ref{Ex1} is dropped, then it is possible to create an example where we not only have everywhere pointwise divergence, but also divergence is obtained on every subset of $\mathcal{I}$ with positive measure.
 
\begin{theorem}\label{Ex2}
 	There exists a continuous function $f:[0,1]\longrightarrow \mathcal{L}^{p}(\OO)$, such that for all measurable sets $T \subset [0,1]$ with $\mu(T)>0$ and every $t\in T$ with Lebesgue density one, there exists a sequence $(t_{n})_{n\in\NN}\subset T$ with $\lim_{n\to\infty}t_{n}=t$ and $\lambda (A_{t})=1$, where $A_{t}=\{x: \lim_{n\to\infty}f_{t_n}(x) \neq f_t(x)\}$
\end{theorem}
\begin{proof}
Let $\{q_{m},m\in\NN\}\subset{\mathcal{I}}\setminus\{0\}$ be dense and assign to each $q_{m}$ a sequence $(s_{m,k})_{k\in\NN}$ defined by
\begin{equation*}
s_{m,k}=q_{m} - \frac{1}{k+r(m)}, \hspace{7pt} \text{where} \hspace{7pt} r(m)=\min\Big\{r:q_{m}-\frac{1}{r}\geq 0\Big\}.
\end{equation*}
Setting $S_{m,k}=[s_{m,k},s_{m,k+1}]$, we note that the vanishing intervals $\{S_{m,k}\}_{k\in\NN}$ partition $[0,q_m]$. To partition the spacial domain, set $$b_{m,k}(t)=\max\Big\{0,\frac{t-s_{m,k}}{s_{m,k+1}-s_{m,k}}-\frac{s_{m,k+1}-t}{4^{k+m}}\Big\}$$ and $$c_{m,k}(t)=\min\Big\{1,\frac{t-s_{m,k}}{s_{m,k+1}-s_{m,k}}+\frac{t-s_{m,k}}{4^{k+m}}\Big\},$$
assigning to every $S_{m,k}$ (possibly empty) intervals $I_{m,k}(t)=[b_{m,k}(t),c_{m,k}(t)]$ of maximal length $\mu(S_{m,k})\times 4^{-(k+m)}$ which emerge, move through $\OO$ at linear speed $\mu(S_{m,k})^{-1}$ and vanish as $t\in\mathcal{I}$ varies. Denoting by $\chi_A$ the characteristic function of a set $A$, we define functions $f^{(m,k)}(t,\cdot):\OO\longrightarrow \mathcal{L}^{p}(\OO)$ by
\begin{equation*}
f^{(m,k)}(t,x)=2^{m}\chi_{S_{m,k}}(t) \chi_{I_{m,k}(t)}(x).
\end{equation*}
These functions satisfy $\|f^{(m,k)}(t,\cdot)\|_{p}\leq\frac{1}{2^{m+k}}$ for all $t \in\mathcal{I}$ and one also checks easily that $f^{(m,k)}(t,\cdot)$ is $\mathcal{L}_{p}$-continuous in the first variable for all $t \in \mathcal{I}$. We can now set 
\begin{equation*}
f_t(x)=\sum_{k,m=1}^{\infty}f^{(m,k)}(t,x)
\end{equation*}
and the limit $f$ is well defined in $\mathcal{L}^p$, since $f^{(m,k)}(t,x)\geq 0$ and  $$\sum_{m,k=1}^{\infty}\|f^{(m,k)}(t,\cdot)\|_{p}<\infty.$$

Let $T\subset{I}$ be of positive measure $\lambda(T)>0$, and let $t\in T$ have density $1$ with respect to $T$. We inductively construct a sequence $(t_{n})_{n\in\NN}\subset T$ with $\lim_{n\to\infty}t_{n}= t$ such that $$\limsup_{n\to\infty} f(t_{n},x)=\infty \neq f(t,x),\textrm{ for almost all }x \in [0,1].$$ To this end, let $(\gamma_{i})_{i\in\NN}\subset(0,1)$ be strictly increasing with limit $1$ and initiate the construction at stage $i=0$ with arbitrary $t_0\in T$ and $n_0=m_0=0$. Assume now, we have completed stages $0,\dots, i-1$ of the construction, i.e. we have chosen the initial members of the sequence $t_{0},\dots,t_{n_{1}},\dots,t_{n_{2}},\dots,t_{n_{i-1}}.$
Since $t$ is a point of density $1$ in $T$, we can fix $\rho_{i}\in(0,t)\setminus\{\frac{1}{l};l\in\NN\}$ such that for every $\rho <\rho_{i}$,
\begin{equation}\label{rhocond} \mu\left(T\cap \left(t-\rho , t\right]\right)\geq \gamma_{i}\rho.\end{equation}
Let now $l\in \mathbb{N}$ be the unique integer with $\frac{1}{l+1}<\rho_{i}<\frac{1}{l}$ and choose $m_i>m_{i-1}$ such that 
\begin{equation}\label{rhocondb} t -\rho_{i} +\frac{1}{l+1} < q_{m_{i}}<t  \text{ and } t-q_{m_{i}}<\frac{\gamma_{i}}{l+1}.\end{equation}
By \eqref{rhocond} and the first inequality of \eqref{rhocondb}, we have 
\begin{align*}\mu\left( T\cap\left[q_{m_{i}}-\frac{1}{l+1},q_{m_{i}}\right]\right)&\geq \mu\left(T\cap \left[q_{m_{i}}-\frac{1}{l+1},t\right]\right)-(t-q_{m_{i}})\\&\geq \gamma_{i}(t-q_{m_{i}}+\frac{1}{l+1})-(t-q_{m_{i}}) ,\end{align*} and by the second inequality of \eqref{rhocondb}
we get
$$\mu\left( T\cap\left[q_{m}-\frac{1}{l+1},q_{m}\right]\right)\geq \frac{\gamma^{2}_{i}}{l+1}.$$
Since the intervals $\{S_{m_i,k};k\in\NN\}$ partition $[q_{m_i}-\frac{1}{l+1},q_{m_i}]$, there must be $k\in\NN$ such that $\mu(S_{m_i,k}\cap T) \geq (s_{m_i,k+1}-s_{m_i,k})\gamma^{2}_{i}$. We denote the index of this interval by $k_i$.\\

For $z>0$ and $J\subset\OO$, we write $zJ=\{za;a\in I\}$ and let also $\mathrm{int}J$ denote the interior of a set $J$. Note that if $r\in \mathrm{int}S_{m,k}$, then $r-s_{m,k}\in(s_{m,k+1}-s_{m,k})\mathrm{int}I_{m,k}(r)$. Using this fact and the scale and translation invariance of Lebesgue measure, we obtain
\begin{equation}\label{final1}\begin{aligned}
&\lambda\left(\bigcup_{r\in S_{m_i,k_i}\cap T}\mathrm{int}I_{m_i,k_i}(r)\right)\\
&=\frac{\lambda\left(\bigcup_{r\in T\cap\mathrm{int}S_{m_i,k_i}}(s_{m_i,k_i+1}-s_{m_i,k_i})\mathrm{int}I_{m_i,k_i}(r)\right)}{(s_{m_i,k_i+1}-s_{m_i,k_i})}\\
&\geq \frac{\lambda\left(\bigcup_{r\in T\cap\mathrm{int}S_{m_i,k_i}}\{r-s_{m,k}\}\right)}{s_{m_i,k_i+1}-s_{m_i,k_i}}=\frac{\lambda\left(\bigcup_{r\in (S_{m_i,k_i}\cap T)}\{r\}\right)}{s_{m_i,k_i+1}-s_{m_i,k_i}}\\
&=\frac{\mu(S_{m_i,k_i}\cap T)}{s_{m_i,k_i+1}-s_{m_i,k_i}}\geq \gamma^{2}_{i}.
\end{aligned}\end{equation}
Since Lebesgue measure is inner regular, we can thus find a compact set $K\subset\bigcup_{r\in S_{m_i,k_i}\cap T}\mathrm{int}I_{m_i,k_i}(r)$ with \begin{equation*}\lambda(K)\geq \gamma_i\lambda\left(\bigcup_{r\in S_{m_i,k_i}\cap T}\mathrm{int}I_{m_i,k_i}(r)\right),\end{equation*}and the compactness of $K$ allows us to select $t_{n_{i-1}+1},\dots,t_{n_{i}}$ from $S_{m_i,k_{m_i}}\cap T$ such that 
\begin{equation*}\lambda\left(\bigcup_{r\in\{t_{n_{i-1}+1},\dots,t_{n_{i}}\}}\mathrm{int}I_{m_i,k_i}(r)\right)\geq \gamma_{i}\lambda\left(\bigcup_{r\in S_{m_i,k_i}\cap T}\mathrm{int}I_{m_i,k_i}(r)\right) 
\end{equation*}
and therefore, combined with \eqref{final1},
\begin{equation*}
\lambda\big\{x: f_{t_{n_{i}+j}}(x)\geq2^{m_i} \textrm{ for at least one } j\in \{1,2,..,n_{i}-n_{i-1}\} \big\} \geq \gamma_{i}^{3},
\end{equation*}
which concludes the construction and finishes the proof, since $(\gamma_i)_{i\in \NN}$ converges to $1$.
\end{proof}

\section{Necessity of discontinuity of $f_t$ in Theorem~\ref{Ex2}}

In this section we prove our final and most general result, namely that dropping continuity with respect to $x$ for almost every $t$ is essential in order to be able to construct an extremely irregular curve like in Theorem~\ref{Ex2}. We show, that if the function $f_t$ is continuous for every $t$, then a refined version of Lusin's theorem in $2$ variables holds. 

\begin{theorem}\label{Lus}
	Let $f:[0,1]\times\OO\longrightarrow\mathbb{R}$ be Borel measurable such that $f_t$ is a continuous function for $\mu$-a.e $t\in\OO$. Then, for every $\eE>0$, there is a set $T_\eE\subset[0,1]$ with $\mu(T_\eE^C)<\eE$ such that the restriction
	$$ f_{|T_\eE\times\OO}:(T_\eE\times\OO,|\cdot|\otimes|\cdot|)\longrightarrow(\RR,|\cdot|)$$
	is a continuous function.
\end{theorem}

Note that in Theorem~\ref{Lus} only the fact that $f$ is a measurable function in $[0,1]\times\OO$ is needed and there is no $\mathcal{L}^p$-continuity assumed. However, as $\mathcal{L}^p$-continuity guarantees that $f$ is a measurable function in $[0,1]\times\OO$, the claimed necessity in Theorem~\ref{Ex2} is a straightforward consequence.

\begin{corollary}
	Let $f$ be a continuous function from $[0,1]$ to $\mathcal{L}^{p}(\Omega)$. If $f_t\in \mathcal{L}^{p}(\Omega)$ is continuous for $\mu$-a.e $t\in[0,1]$, then the assertion in Theorem~\ref{Lus} holds.
\end{corollary}

Before we proceed with the proof of Theorem~\ref{Lus} we would like to point out the difference between Theorem~\ref{Lus} and the classical result of Lusin. In Lusin's theorem an arbitrarily small set of $\mathcal{I}\times\OO$ is removed in order to establish continuity on the remainder. In our case the stronger assumption of continuity with respect to one variable entails the information that this small set is of the form $T_{\eE}^C\times\OO$, so it is only necessary to remove a ''slice`` in the space-time domain.\\

For the proof of Theorem~\ref{Lus} we also need the following preliminary result.

\begin{lemma}\label{eqcon}
	Let $F=\{f_{t}; t\in\mathcal{I}\}$ be a family of continuous functions. Then, for every $\eE>0$, there is a set $S_{\eE}\subset\mathcal{I}$, such that $\mu(S_{\eE}^C)<\eE$, with the property that $F_\eE:=\{f_{t}; t\in S_{\eE}\}$ is equicontinuous.
\end{lemma}

\begin{proof}
	Let $\oO_\delta$ denote the $\dD$-oscillation functional,
	\begin{equation}\label{omega}
	\omega_\dD(g):= \sup\{|g(x)-g(y)|:\:|x-y|< \delta\}
	\end{equation}
	and set $\omega_{n}(t)=\omega_{\frac{1}{n}}(f_t)$ on $F$. We have $\lim_{n\Ra\infty}\omega_{n}(t)=0$ for every $t$, since all $f_t$ are continuous. From Egorov's Theorem (see, e.g. \cite[Theorem~8.3]{Ox71}) we deduce that for every $\eE>0$, there exists a set $S_{\eE}$ with $\mu(S_{\epsilon}^C)<\eE$ such that $\omega_{n|\:{S_\eE}}$ converges uniformly. Now, fixing $\eE>0$, we wish to show that the uniform convergence of $\omega_{n}$ to zero implies equicontinuity of $F_{\eE}$.
 
	Let $\eta>0$. Since $\lim_{n\to\infty}\omega_{n}= 0$ uniformly on $S_{\eE}$, there exists $n_{0}\in\NN$ such that for all $n>n_{0}$, $$\oO_{n}<\eta\textrm{ for every }t \in S_{\eE}.$$ Hence, choosing $\dD=\frac{1}{n_0}$ in \eqref{omega} and evaluating $\oO_\dD$ on $F_\eE$, we obtain
	\begin{displaymath}
		\sup\{|f_t(x)-f_t(y)|:\: |x-y|<\delta\}<\eta. 
	\end{displaymath}
	This means $F_\eE$ is equicontinuous, since $\eta$ was choosen arbitrarily.
\end{proof}

From Lemma~\ref{eqcon} and Lusin's Theorem we can finally deduce Theorem~\ref{Lus}.
\begin{proof}[Proof of Theorem \ref{Lus}]

Since $f$ is Borel measurable, we have that $f^{x},$ where $f^x(t):=f(t,x)$, is a Borel measurable function for every $x\in[0,1]$. We can therefore choose a dense countable subset $X=\{x_{n};n\in\NN\}\subset\OO$ such that $f^{x_{n}}$ is Borel measurable for every $n\in\NN$. By Lusin's theorem, for every function $f^{x_{n}},n\in\NN$, and any fixed small parameter $\eE$ there is a set $U_{n}\subset [0,1]$ such that $\mu((U_{n})^C)<\frac{\eE}{2^{n+1}}$ and $f^{x_{n}}_{|_{U_n}}$ is continuous. Now define $V_{\frac{\eE}{2}}:=\bigcap_{n=1}^{\infty}U_n$ and apply Lemma~\ref{eqcon} to the family $\{f(t,\cdot);t\in[0,1]\}$ to obtain a set $S_{\frac{\eE}{2}}$ such that $\{f(t,\cdot);t\in S_{\frac{\eE}{2}}\}$ is equicontinuous. Now, $T_\eE:=S_{\frac{\eE}{2}}\cap V_{\frac{\eE}{2}}$, is the desired set. It remains to prove that the restriction $f_{|T_\varepsilon\times\OO}$ is a continuous function, and we are going to use the sequential definition of continuity to do so. 

Let $(t_{0},y_{0})\in T_{\varepsilon}\times\OO$. Let also a sequence $(t_n,y_n)\in T_{\varepsilon}\times\OO$ such that  $\lim_{n\rightarrow\infty}(t_{n},y_{n})=(t_{0},y_{0}).$ Finally let $\eta>0$. By equicontinuity of $\{f(t,\cdot);t\in T_{\varepsilon}\}$ there exists a $\delta>0$
such that 
$$|f(t,y')-f(t,y)|<\frac{\eta}{3} \mbox{ for all } t \in T_{\epsilon} \mbox{ and } y',y\in\OO \mbox{ with }|y'-y|<\delta.$$ 
By density of $X$ in $\OO$, there exists a $x_{0}$ such that 
$|y_{0}-x_{0}|<\frac{\delta}{2}$.
Since $y_n\rightarrow y_{0}$, we can find a $n_{1}$ such that $|y_{n}-y_{0}|<\frac{\delta}{2}$ and thus $|y_{n}-x_{0}|<\delta
,\forall n\in \NN$ with $n>n_{1}.$
Furthermore by continuity of $f(\cdot,x_{0})$ in $T_{\epsilon}$ there exists a $n_{2}$ such that
$\forall n>n_{2}$ we have  $\|f(t_{n},x_{0})-f(t_{0},x_{0})\|<\frac{\eta}{3}$ 

Now, for $n>max\{ n_{1},n_{2}\}$ we have
\begin{align*}
\|f(t_n,y_n)-f(t_0,y_0)\|&\leq\|f(t_n,y_n)-f(t_n,x_0)\|+\|f(t_n,x_0)-f(t_0,x_0)\|\\
&+\|f(t_0,x_0)-f(t_0,y_0)\|\leq\frac{\eta}{3}+\frac{\eta}{3}+\frac{\eta}{3}=\eta
\end{align*}

\end{proof}

\end{document}